\documentclass[a4paper]{amsart}
\usepackage{amscd}
\usepackage{amsthm}
\usepackage{graphicx}
\usepackage{amsmath}
\usepackage{amsfonts}
\usepackage{amssymb}

\newtheorem{theorem}{Theorem}

\newtheorem{corollary}[theorem]{Corollary}

\newtheorem{definition}[theorem]{Definition}

\newtheorem{lemma}[theorem]{Lemma}

\newtheorem{proposition}[theorem]{Proposition}
\newtheorem{remark}[theorem]{Remark}

\begin{document}
\title[Calder\'on-Hardy Spaces with variable exponents ]{Calder\'on-Hardy Spaces with variable exponents and the Solution of the Equation $\Delta^{m} F = f$ for $f \in H^{p(\cdot)}(\mathbb{R}^{n})$}
\author{Pablo Rocha}
\address{Universidad Nacional del Sur, INMABB (Conicet), Bah\'{\i}a Blanca, 8000 Buenos Aires, Argentina}
\email{pablo.rocha@uns.edu.ar}
\thanks{\textbf{Key
words and phrases}: Variable Calder\'on-Hardy Spaces, Variable Hardy Spaces, Atomic decomposition}
\thanks{\textbf{2.010
Math. Subject Classification}: 42B25, 42B30}
\thanks{The author is partially supported by SECYTUNC and UNS}
\maketitle
\begin{abstract}
In this article we define the Calder\'on-Hardy spaces with variable exponents on $\mathbb{R}^{n}$, $\mathcal{H}^{p(.)}_{q, \gamma}(\mathbb{R}^{n})$, and we show that for $m \in \mathbb{N}$ the operator $\Delta^{m}$ is a bijective mapping from $\mathcal{H}^{p(.)}_{q, 2m}(\mathbb{R}^{n})$ onto $H^{p(.)}(\mathbb{R}^{n})$.
\end{abstract}

\section{Introduction}

Given a measurable function $p(\cdot) : \mathbb{R}^{n} \rightarrow (0, \infty)$ such that $$0 < \inf_{x \in \mathbb{R}^{n}} p(x) \leq \sup_{x \in \mathbb{R}^{n}} p(x) < \infty,$$
let $L^{p(\cdot)}(\mathbb{R}^{n})$ denote the space of all measurable functions $f$ such that for some $\lambda > 0$,
$$ \int_{\mathbb{R}^{n}} \, \left| \frac{f(x)}{\lambda} \right|^{p(x)} \, dx < \infty.$$
We set $$\| f \|_{p(\cdot)} = \inf \left\{ \lambda > 0 : \int_{\mathbb{R}^{n}} \, \left| \frac{f(x)}{\lambda} \right|^{p(x)} \, dx \leq 1 \right\}.$$
We see that $\left( L^{p(\cdot)}(\mathbb{R}^{n}), \| . \|_{p(\cdot)} \right)$ is a quasi normed space.
These spaces are referred to as the Lebesgue spaces with variable exponents. In the last years many authors have extended the machinery of classical harmonic analysis to these spaces. See, for example \cite{capone}, \cite{uribe}, \cite{diening}, \cite{ruzicka}, \cite{kova}.

In the celebrate paper \cite{fefferman}, C. Fefferman and E. M. Stein defined the Hardy Spaces $H^{p}(\mathbb{R}^{n})$, $0 < p < \infty$, with the norm given by
$$\| f \|_{H^{p}} = \left\| \sup_{t>0} \sup_{\phi \in \mathcal{F}_{N}} | t^{-n} \phi(t^{-1}.) \ast f| \right\|_{p},$$
for suitable family $\mathcal{F}_{N}$. In the paper \cite{nakai}, E. Nakai and Y. Sawano defined the Hardy spaces with variable exponents $H^{p(.)}(\mathbb{R}^{n})$, replacing $L^{p}$ by $L^{p(.)}$ in the above norm and they investigate their several properties.

Let $L^{q}_{loc}(\mathbb{R}^{n})$, $1 < q < \infty$, be the space of all measurable functions $f$ on $\mathbb{R}^{n}$ that belong locally to $L^{q}$ for compact sets of $\mathbb{R}^{n}$. We endowed $L^{q}_{loc}(\mathbb{R}^{n})$ with the topology generated by the seminorms
$$|f|_{q, \, Q} = \left( |Q|^{-1} \int_{Q} \, |f(x)|^{q}\, dx \right)^{1/q},$$ where $Q$ is a cube in $\mathbb{R}^{n}$ and $|Q|$ denotes its Lebesgue measure.

For $f \in L^{q}_{loc}(\mathbb{R}^{n})$, we define a maximal function $\eta_{q, \, \gamma}(f; x)$ as
$$\eta_{q, \, \gamma}(f; x) = \sup_{r > 0} r^{-\gamma} |f|_{q, \, Q(x, r)},$$
where $\gamma$ is a positive real number and $Q(x, r)$ is the cube centered at $x$ with side length $r$.

Let $k$ a non negative integer and $\mathcal{P}_{k}$ the subspace of $L^{q}_{loc}(\mathbb{R}^{n})$ formed by all the polynomials of degree at most $k$. We denote by $E^{q}_{k}$ the quotient space of $L^{q}_{loc}(\mathbb{R}^{n})$ by $\mathcal{P}_{k}$. If $F \in E^{q}_{k}$, we define the seminorm
$\| F \|_{q, \, Q} = \inf \left\{ |f|_{q, \, Q} : f \in F \right\}$. The family of all these seminorms induces on $E^{q}_{k}$ the quotient topology.

Given a positive real number $\gamma$, we can write $\gamma = k + t$, where $k$ is a non negative integer and $0 < t \leq 1$. This decomposition is unique.

For $F \in E^{q}_{k}$, we define a maximal function $N_{q, \, \gamma}(F; x)$ as $$N_{q, \, \gamma}(F; x) = \inf \left\{ \eta_{q, \, \gamma}(f; x) : f \in F \right\}.$$ This type of maximal function was introduced by A. P. Calder\'on in \cite{calderon}.

Let $p(.) : \mathbb{R}^{n} \rightarrow (0, \infty)$ be a measurable function such that $0< p_{-} \leq p_{+} < \infty$, we say that an element $F \in E^{q}_{k}$ belongs to the Calder\'on-Hardy space with variable exponents $\mathcal{H}^{p(.)}_{q, \, \gamma}(\mathbb{R}^{n})$ if the maximal function $N_{q, \, \gamma}(F; \cdot) \in L^{p(.)}(\mathbb{R}^{n})$. The "norm" of $F$ in $\mathcal{H}^{p(.)}_{q, \, \gamma}(\mathbb{R}^{n})$ is defined as $\| F \|_{\mathcal{H}^{p(.)}_{q, \, s}(\mathbb{R}^{n})} = \| N_{q, \, \gamma}(F; \cdot) \|_{p(.)}$.

The Calder\'on-Hardy spaces were defined in the setting of the classical Lebesgue spaces by A. B. Gatto, J. G. Jim\'enez and C. Segovia in \cite{segovia}, they characterize the solutions of $\Delta^{m} F = f$, $m \in \mathbb{N}$, for $f \in H^{p}(\mathbb{R}^{n})$. Moreover, they proved that the operator $\Delta^{m}$ is a bijective mapping from the Calder\'on-Hardy spaces onto $H^{p}(\mathbb{R}^{n})$. In this work we show that this result holds in the context of the Lebesgue spaces with variable exponents. In \cite{sheldy}, S. Ombrosi studied the weighted version of the Calder\'on-Hardy spaces. The Calder\'on-Hardy spaces were mentioned for first time with this name in \cite{sheldy2}. A. Perini studied in \cite{perini} the boundedness of one-sided fractional integrals on these spaces. In \cite{sheldy1}, S. Ombrosi, A. Perini and R. Testoni obtain a complex interpolation theorem between weighted Calder\'on-Hardy spaces for weights in a Sawyer class.

Given a function $p(\cdot):\mathbb{R}^{n}\rightarrow(0,\infty)$ we say that $p(\cdot)$ is
locally log-H\"{o}lder continuous, and denote this by $p(\cdot)\in
LH_{0}(\mathbb{R}^{n})$, if there exists a positive constant $C_{0}$ such that
\[
\left\vert p(x)-p(y)\right\vert \leq\frac{C_{0}}{-\log\left\vert
x-y\right\vert }, \,\,\, \left\vert
x-y\right\vert <\frac{1}{2}.
\]
We say that $p(\cdot)$ is log-H\"{o}lder continuous at infinity, and denote
this by $p(\cdot)\in LH_{\infty}(\mathbb{R}^{n})$ if there exists a positive constant
$C_{\infty}$ such that
\[
\left\vert p(x)-p(y)\right\vert \leq\frac{C_{\infty}}{\log\left(  e+\left\vert
x\right\vert \right)  }, \,\,\, \left\vert
y\right\vert \geq\left\vert x\right\vert .
\]
We note that the condition log-H\"{o}lder continuous at infinity is equivalent to the existence of constants
$C_{\infty}$ and $p_{\infty}$ such that
\[
\left\vert p(x)-p_{\infty}\right\vert \leq\frac{C_{\infty}}{\log\left(  e+\left\vert
x\right\vert \right)  }, \,\,\, x \in \mathbb{R}^{n}.
\]
As usual we will denote $p_{+} = \sup_{x \in \mathbb{R}^{n}} p(x)$, $p_{-}= \inf_{x \in \mathbb{R}^{n}} p(x)$, $\underline{p} = \min \{ p_{-}, 1\}$ and $\Delta$ stands for the Laplacian.

Let $F \in E^{q}_{2m-1}$ and $f \in F$. Since $f$ belongs to $L^{q}_{loc}(\mathbb{R}^{n})$, $\Delta^{m}f$ is defined in sense of distributions. On the other hand, since any two representatives of $F$ differ in a polynomial of degree smaller than $2m$, we get that $\Delta^{m}f$ is independent of the representative $f \in F$ chosen. Therefore, for $F \in E^{q}_{2m-1}$, we shall define $\Delta^{m}F$ as the distribution $\Delta^{m}f$, where $f$ is any representative of $F$.

\qquad

Our main result is contained in the following

\begin{theorem} Let $p(.)$ be a function that belongs to $LH_{0}(\mathbb{R}^{n})\cap LH_{\infty}(\mathbb{R}^{n})$, $1 < q < \infty$ and $m \in \mathbb{N}$ such that $0 < p_{-} \leq p_{+} < \infty$ and $n (2m + n/q)^{-1} < \underline{p}$. Then for $q$ sufficiently large the operator $\Delta^{m}$ is a bijective mapping from $\mathcal{H}^{p(.)}_{q, 2m}(\mathbb{R}^{n})$ onto $H^{p(.)}(\mathbb{R}^{n})$. Moreover, there exist two positive constant $c_1$ and $c_2$ such that
\[
c_1 \|F \|_{\mathcal{H}^{p(.)}_{q, 2m}} \leq \| \Delta^{m}F \|_{H^{p(.)}} \leq c_2 \|F \|_{\mathcal{H}^{p(.)}_{q, 2m}}
\]
hold for all $F \in \mathcal{H}^{p(.)}_{q, 2m}(\mathbb{R}^{n})$.
\end{theorem}

The case $p_{+} \leq n (2m + n/q)^{-1}$ is trivial.

\begin{theorem} If $p(.)$ is a positive measurable function such that $p_{+} \leq n (2m + n/q)^{-1}$, then $\mathcal{H}^{p(.)}_{q, \, 2m}(\mathbb{R}^{n}) = \{ 0 \}.$
\end{theorem}

In Section 2 we state some auxiliary lemmas and propositions to get the main results. We also recall the definition and atomic decomposition of the Hardy
spaces with variable exponents given in \cite{nakai}. In Section 3 we give the proofs of Theorem 1 and Theorem 2.

\qquad

\textbf{Notation:} The symbol $A\lesssim B$ stands for the inequality $A\leq
cB$ for some constant $c.$ We denote by $Q\left( x_0,r\right) $ the cube centered at $x_0 \in \mathbb{R}^{n}$ with side lenght $r.$ Given a cube $Q=Q\left( x_0,r\right),$ we set $\delta Q=Q(x_0, \delta r)$ and $l\left( Q\right) =r.$ For a measurable subset $E\subseteq \mathbb{R}^{n}$ we denote by $\left\vert E\right\vert $ and $\chi_{E}$ the Lebesgue measure of $E$ and the characteristic function of $E$ respectively. For a function $p(.):\mathbb{R}^{n}\rightarrow \left( 0,\infty
\right) $ we define $p_{+} = \sup_{x \in \mathbb{R}^{n}} p(x)$, $p_{-}= \inf_{x \in \mathbb{R}^{n}} p(x)$ and $\underline{p}=\min \left\{ p_{-},1\right\}$. As usual we denote with $S(\mathbb{R}^{n})$ the space of smooth and rapidly decreasing functions, with $S^{\prime }(\mathbb{R}^{n})$ the dual space and $\Delta$ stands for the Laplacian. If $\mathbf{\alpha }$ is the multiindex $\alpha =(\alpha_{1},...,\alpha_{n})$
then $\left\vert \alpha \right\vert =\alpha _{1}+...+\alpha_{n}.$

Throughout this paper, $c$ will denote a positive constant, not necessarily the
same at each occurrence.

\section{Preliminaries}

The function $p(.): \mathbb{R}^{n} \rightarrow (0, \infty)$ is called the variable exponent. Here we adopt the standard notation in variable exponents. We write
$$p_{-}= \inf_{x \in \mathbb{R}^{n}} p(x), \,\,\,\,\,\, p_{+} = \sup_{x \in \mathbb{R}^{n}} p(x), \,\,\,\, \text{and} \,\,\,\, \underline{p}=\min \left\{ p_{-},1\right\}.$$
Recall that we assumed $0 < p_{-} \leq p_{+} < \infty$.

For measurable function $f$, let
$$\| f \|_{p(\cdot)} = \inf \left\{ \lambda > 0 : \int_{\mathbb{R}^{n}} \, \left| \frac{f(x)}{\lambda} \right|^{p(x)} \, dx \leq 1 \right\},$$
it not so hard to see the following

$1.$ $\| f \|_{p(\cdot)} \geq 0$, and $\| f \|_{p(\cdot)}=0$ if and only if $f \equiv 0$.

$2.$ $\| c \, f \|_{p(\cdot)} = |c| \, \| f \|_{p(\cdot)}$ for $c \in \mathbb{C}$.

$3.$ $\| f + g \|_{p(\cdot)}^{\underline{p}} \leq \| f \|_{p(\cdot)}^{\underline{p}} + \| g \|_{p(\cdot)}^{\underline{p}}$. \\
$ \,\, $ \\
A direct consequence of $\underline{p}$-triangle inequality is the quasi-triangle inequality
$$\| f + g \|_{p(\cdot)} \leq 2^{1/\underline{p}}\left(\| f \|_{p(\cdot)} + \| g \|_{p(\cdot)}\right),$$
for all $f, \, g \in L^{p(\cdot)}(\mathbb{R}^{n})$.

\qquad

The following lemmas are crucial to get the principal result.

\begin{lemma} $($Lemma 6 in \cite{calderon}$)$ The maximal function $N_{q; \, \gamma}(F; x)$ associated with a class $F$ in $E_{k}^{q}$ is lower semicontinuous.
\end{lemma}

\begin{lemma} Let $F \in E^{q}_{k}$ with $N_{q, \, \gamma}(F; x_0) < \infty,$ for some $x_0 \in \mathbb{R}^{n}$. Then:

\qquad

$(i)$ There exists a unique $f \in F$ such that $\eta_{q, \, \gamma} (f; x_0) < \infty$ and, therefore, $\eta_{q, \, \gamma} (f; x_0) = N_{q, \, \gamma}(F; x_0)$.

$(ii)$ For any cube $Q$, there is a constant $c$ depending on $x_0$ and $Q$ such that if $f$ is the unique representative of $F$ given in $(i)$, then
$$\|F\|_{q, \, Q} \leq |f|_{q, \, Q} \leq c \, \eta_{q, \, \gamma} (f; x_0) = c \, N_{q, \, \gamma}(F; x_0).$$

The constant $c$ can be chosen independently of $x_0$ provided that $x_0$ varies in a compact set.
\end{lemma}
\begin{proof} The proof is similar to that of Lemma 3 in \cite{segovia}.
\end{proof}

\begin{corollary} If $\{ F_{j} \}$ is a sequence of elements of $E^{q}_{k}$ converging to $F$ in $\mathcal{H}^{p(\cdot)}_{q, \, \gamma}(\mathbb{R}^{n})$,  $0 < p_{-} \leq p_{+} < \infty$, then $\{ F_{j} \}$ converges to $F$ in $E^{q}_{k}$.
\end{corollary}

\begin{proof} For any cube $Q$, by $(ii)$ of Lemma 3, we have
$$\| F- F_{j} \|_{q, \, Q} \leq c \, \| \chi_{Q} \|_{p(.)}^{-1} \| \chi_{Q} \,\, N_{q, \, \gamma}(F - F_{j}; \cdot ) \|_{p(\cdot)} \leq c \, \| F - F_{j} \|_{\mathcal{H}^{p(\cdot)}_{q, \, \gamma}},$$
which proves the corollary.
\end{proof}

\begin{lemma} Let $\{ F_{j} \}$ be a sequence in $E^{q}_{k}$ such that the series $\sum_j N_{q, \, \gamma}(F_{j}; \, x )$ is finite p.p.x in $\mathbb{R}^{n}$. Then

\qquad

$(i)$ The series $\sum_j F_j$ converges in $E_{k}^{q}$ to an element $F$ and $$N_{q, \, \gamma}(F; \, x ) \leq \sum_j N_{q, \, \gamma}(F_{j}; \, x ), \,\,\, p.p.x\in \mathbb{R}^{n}.$$

$(ii)$ Let $x_0$ be a point where $\sum_j N_{q, \, \gamma}(F_{j}; \, x_0 )$ is finite. If $f_j$ is the unique representative of $F_j$ satisfying
$\eta_{q, \, \gamma} (f_j; x_0) = N_{q, \, \gamma}(F_j; x_0)$, then $\sum_j f_j$ converges in $L^{q}_{loc}(\mathbb{R}^{n})$ to a function $f$ that is the unique representative of $F$ satisfying $\eta_{q, \, \gamma} (f; x_0) = N_{q, \, \gamma}(F; x_0)$
\end{lemma}

\begin{proof} The proof is similar to that of Lemma 4 in \cite{segovia}.
\end{proof}

On the set of the all measurable function $f$ we define the modular function $\rho_{p(.)}$ by
$$\rho_{p(.)}(f) = \int_{\mathbb{R}^{n}} \, |f(x)|^{p(x)} \, dx.$$
It is well known that if $0 < p_{-} \leq p_{+} < \infty$ and $0 \neq f \in L^{p(\cdot)}(\mathbb{R}^{n})$, then
$$\rho_{p(.)}\left(\|f\|_{p(.)}^{-1} \, f \right) = \int_{\mathbb{R}^{n}} \, \left(\frac{|f(x)|}{\|f\|_{p(.)}} \right)^{p(x)} \, dx=1.$$

\begin{lemma} Let $p(.)$ a measurable function such that $0 < p_{-} \leq p_{+} < \infty$. Then $f\in L^{p(.)}(\mathbb{R}^{n})$
if and only if $\rho_{p(.)}(f) < \infty$.
\end{lemma}

\begin{proof} Clearly, if $\rho_{p(.)}(f) < \infty$, then $f\in L^{p(.)}(\mathbb{R}^{n})$. Conversely, if $f\in L^{p(.)}(\mathbb{R}^{n})$, then we have that
$\rho_{p(.)}(f / \lambda) < \infty$ for some $\lambda >1$. Then
$$\rho_{p(.)}(f)= \int_{\mathbb{R}^{n}} \, \left| \frac{\lambda f(x)}{\lambda} \right|^{p(x)} \, dx \leq \lambda^{p_{+}} \rho_{p(.)}(f / \lambda) < \infty.$$
\end{proof}

\begin{lemma} Let $p(.)$ a measurable function such that $0 < p_{-} \leq p_{+} < \infty$. If $\{ f_j \}$ is a sequence of measurable functions such that $\rho_{p(.)}(f_j) \rightarrow 0$, then $\| f_j \|_{p(.)} \rightarrow 0$.
\end{lemma}

\begin{proof} Suppose that $\rho_{p(.)}(f_j) \rightarrow 0$. Given $0 < \epsilon < 1$ for sufficiently large $j$ we have $\rho_{p(.)}(f_j) \leq \epsilon$ and so
\[
\rho_{p(.)}\left(f_j\rho_{p(.)}(f_j)^{-1/p_{+}}\right) \leq \rho_{p(.)}(f_j)^{-1} \rho_{p(.)}(f_j) =1,
\]
from this it follows that $\| f_j \|_{p(.)} \leq \rho_{p(.)}(f_j)^{1/p_{+}} \leq \epsilon^{1/p_{+}}$. Thus, $\| f_j \|_{p(.)} \rightarrow 0$.
\end{proof}

\begin{proposition} The space $\mathcal{H}^{p(\cdot)}_{q, \, \gamma}(\mathbb{R}^{n})$, $0 < p_{-} \leq p_{+} < \infty$, is complete.
\end{proposition}

\begin{proof} It is enough to show that $\mathcal{H}^{p(\cdot)}_{q, \, \gamma}$ has the Riesz-Fisher property: given any sequence $\{ F_j \}$ in $\mathcal{H}^{p(\cdot)}_{q, \, \gamma}$ such that $$\sum_{j} \| F_j \|_{\mathcal{H}^{p(\cdot)}_{q, \, \gamma}}^{\underline{p}} < \infty,$$
the series $\sum_{j} F_j$ converges in $\mathcal{H}^{p(\cdot)}_{q, \, \gamma}$. \\
Let $1 \leq l$ be fixed, then $$\left\| \sum_{j=l}^{k} N_{q, \, \gamma}(F_{j}; \, . ) \right\|_{p(.)}^{\underline{p}} \leq \sum_{j=l}^{k} \left\| N_{q, \, \gamma}(F_{j}; \, . ) \right\|_{p(.)}^{\underline{p}} \leq \sum_{j=l}^{\infty} \| F_j \|_{\mathcal{H}^{p(\cdot)}_{q, \, \gamma}}^{\underline{p}} =: \alpha_l < \infty,$$
for all $k \geq l$, thus $$\int_{\mathbb{R}^{n}} \, \left( \alpha_{l}^{-1/ \underline{p}} \, \sum_{j=l}^{k} N_{q, \, \gamma}(F_{j}; \, x ) \right)^{p(x)} \, dx$$  $$\leq \int_{\mathbb{R}^{n}} \left( \left\| \sum_{j=l}^{k} N_{q, \, \gamma}(F_{j}; \, . ) \right\|_{p(.)}^{-1} \, \sum_{j=l}^{k} N_{q, \, \gamma}(F_{j}; \, x ) \right)^{p(x)} \, dx =1, \,\,\, \forall \, k \geq l$$ it follows from Fatou's lemma as $k \rightarrow \infty$ that
$$\int_{\mathbb{R}^{n}} \, \left( \alpha_{l}^{-1/ \underline{p}} \, \sum_{j=l}^{\infty} N_{q, \, \gamma}(F_{j}; \, x ) \right)^{p(x)} \, dx \leq 1$$
thus
\begin{equation}
\left\|  \sum_{j=l}^{\infty} N_{q, \, \gamma}(F_{j}; \, \cdot) \right\|_{p(\cdot)}^{\underline{p}} \leq \alpha_{l} = \sum_{j=l}^{\infty} \| F_j \|_{\mathcal{H}^{p(\cdot)}_{q, \, \gamma}}^{\underline{p}} < \infty, \,\,\,\, \forall \, l \geq 1  \label{serie}.
\end{equation}
Taking $l=1$ in (\ref{serie}), from Lemma 7 it follows that $\sum_{j} N_{q, \, \gamma}(F_{j}; \, x)$ is finite p.p.$x \in \mathbb{R}^{n}$. Then, by $(i)$ of Lemma 6, the series $\sum_j F_j$ converges in $E_{k}^{q}$ to an element $F$. Now $$N_{q, \, \gamma}\left( F - \sum_{j=1}^{k} F_j; x \right) \leq \sum_{j=k+1}^{\infty} N_{q, \, \gamma} (F_j; x),$$
from this and (\ref{serie}) we get
$$\left\| F - \sum_{j=1}^{k} F_j \right\|_{\mathcal{H}^{p(\cdot)}_{q, \, \gamma}}^{\underline{p}} \leq \sum_{j=k+1}^{\infty} \| F_j \|_{\mathcal{H}^{p(\cdot)}_{q, \, \gamma}}^{\underline{p}},$$
and since the right-hand side tends to $0$ as $k \rightarrow \infty$, the series $\sum_{j}F_j$ converges to $F$ in $\mathcal{H}^{p(\cdot)}_{q, \, \gamma}(\mathbb{R}^{n})$.
\end{proof}

In the paper \cite{nakai}, E. Nakai and Y. Sawano give a variety of distinct
approaches, based on differing definitions, all lead to the same notion of the
variable Hardy space $H^{p(.)}.$

We recall the definition and the atomic decomposition of the Hardy spaces with variable exponents.

\qquad

Topologize $\mathcal{S}(\mathbb{R}^{n})$ by the collection of semi-norms $\{ p_{N} \}_{N \in \mathbb{N}}$ given by
$$p_{N}(\varphi) = \sum
\limits_{\left\vert \mathbf{\beta}\right\vert \leq N}\sup\limits_{x\in
\mathbb{R}^{n}}\left(  1+\left\vert x\right\vert \right)  ^{N}\left\vert
\partial^{\mathbf{\beta}}\varphi(x)\right\vert,$$
for each $N \in \mathbb{N}$. We set $\mathcal{F}_{N}=\left\{  \varphi\in \mathcal{S}(\mathbb{R}^{n}): p_{N}(\varphi) \leq 1 \right\}$. Let $f \in \mathcal{S}'(\mathbb{R}^{n})$, we denote by $\mathcal{M}_{\mathcal{F}_{N}}$ the grand maximal
operator given by
\[
\mathcal{M}_{\mathcal{F}_{N}}f(x)=\sup\limits_{t>0}\sup\limits_{\varphi\in\mathcal{F}_{N}%
}\left\vert \left(  t^{-n}\varphi(t^{-1} \cdot)\ast f\right)  \left(  x\right)
\right\vert ,
\]
where $N$ is a large and fix integer.

\qquad

The Hardy space with variable exponents $H^{p(\cdot)}(\mathbb{R}^{n})$ is the set of all $f \in S^{\prime}
(\mathbb{R}^{n})$ for which $\mathcal{M}_{\mathcal{F}_{N}}f \in L^{p(\cdot)}(\mathbb{R}^{n}).$ In this case we define $\left\Vert f\right\Vert _{H^{p(\cdot)}
}=\left\Vert \mathcal{M}_{\mathcal{F}_{N}}f\right\Vert _{p(\cdot)}$.

\qquad

Let $\phi \in \mathcal{S}(\mathbb{R}^{n})$ be a function such that $\int \phi(x) dx \neq 0$. For $f \in \mathcal{S}'(\mathbb{R}^{n})$, we define the maximal function $M_{\phi}f$ by
$$M_{\phi}f(x)= \sup_{t>0} \left\vert \left(  t^{-n}\phi(t^{-1} \, \cdot)\ast f\right)  \left(  x\right)\right\vert.$$
Theorem 1.2 in \cite{nakai} asserts that the quantities $\| M_{\phi}f \|_{p(\cdot)}$ and $\left\Vert \mathcal{M}_{\mathcal{F}_{N}}f\right\Vert _{p(\cdot)}$ are comparable, with bounds independent of $f$ if $N$ is sufficiently large.

\begin{definition}
$((p(\cdot),p_{0},d)-atom)$. Let $p(\cdot):\mathbb{R}^{n}\rightarrow \left( 0,\infty
\right) $, $0<p_{-}\leq p_{+}<p_{0}\leq \infty $ and $p_{0}\geq 1.$ Fix an
integer $d\geq d_{p(\cdot)}=\min \left\{ l\in \mathbb{N\cup }\left\{ 0\right\}
:p_{-}(n+l+1)>n\right\} .$ A function $a$ on $\mathbb{R}^{n}$ is called an $
(p(\cdot),p_{0},d)-$atom if there exists a cube $Q$ such that\newline
$a_{1})$ $\textit{supp}\left( a\right) \subset Q,$\newline
$a_{2})$ $\left\Vert a\right\Vert _{p_{0}}\leq \frac{\left\vert Q\right\vert
^{\frac{1}{p_{0}}}}{\left\Vert \chi _{Q}\right\Vert _{p(\cdot)}},$\newline
$a_{3})$ $\int a(x)x^{\alpha }dx=0$ for all $\left\vert \alpha \right\vert
\leq d.$
\end{definition}

\begin{remark} Let $a$ be an $(p(\cdot),p_{0},d)-atom$ and $1 < s < p_0$, then H\"older's inequality implies $\| a \|_{s} \leq \frac{|Q|^{1/s}}{\| \chi_{Q} \|_{p(\cdot)}}.$
\end{remark}

\begin{definition}
For sequences of nonnegative numbers $\left\{ k_{j}\right\} _{j=1}^{\infty }$
and cubes $\left\{ Q_{j}\right\} _{j=1}^{\infty }$ and for a function $p(\cdot):%
\mathbb{R}^{n}\rightarrow \left( 0,\infty \right) $, we define
\[
\mathcal{A}\left( \left\{ k_{j}\right\} _{j=1}^{\infty },\left\{
Q_{j}\right\} _{j=1}^{\infty },p(\cdot)\right) =\left\Vert \left\{
\sum\limits_{j=1}^{\infty }\left( \frac{k_{j}\chi _{Q_{j}}}{\left\Vert \chi
_{Q_{j}}\right\Vert _{p(\cdot)}}\right) ^{\underline{p}}\right\} ^{\frac{1}{%
\underline{p}}}\right\Vert _{p(\cdot)}.
\]%
\newline
The space $H_{atom}^{p(\cdot),p_{0},d}\left( \mathbb{R}^{n}\right) $ is the set
of all distributions $f\in S^{\prime }(\mathbb{R}^{n})$ such that it can be
written as
\begin{equation}
f=\sum\limits_{j=1}^{\infty }k_{j}a_{j}  \label{desc. atomica}
\end{equation}%
in $S^{\prime }(\mathbb{R}^{n}),$ where $\left\{ k_{j}\right\}
_{j=1}^{\infty }$ is a sequence of non negative numbers, the $a_{j}`s$ are $%
(p(\cdot),p_{0},d)-$atoms and $\mathcal{A}\left( \left\{ k_{j}\right\}
_{j=1}^{\infty },\left\{ Q_{j}\right\} _{j=1}^{\infty },p(\cdot)\right) <\infty
. $ One defines
\[
\left\Vert f\right\Vert _{H_{atom}^{p(\cdot),p_{0},d}}=\inf \mathcal{A}\left(
\left\{ k_{j}\right\} _{j=1}^{\infty },\left\{ Q_{j}\right\} _{j=1}^{\infty
},p(\cdot)\right)
\]%
where the infimun is taken over all admissible expressions as in (\ref{desc.
atomica}).
\end{definition}

Theorem 4.6 in \cite{nakai} asserts that the quantities $\left\Vert f\right\Vert
_{H_{atom}^{p(\cdot),p_{0},d}}$ and $\left\Vert f\right\Vert _{H^{p(\cdot)}}$ are comparable. Moreover, $f$ admits an atomic decomposition $f=\sum\limits_{j=1}^{\infty }k_{j}a_{j}$ such that $$\mathcal{A}\left( \left\{ k_{j}\right\}
_{j=1}^{\infty },\left\{ Q_{j}\right\} _{j=1}^{\infty },p(\cdot)\right) \leq c \, \|f \|_{H^{p(\cdot)}}.$$

\qquad

Let $h$ be the function defined by
\begin{equation}
h(x) = \left\{ \begin{array}{cc}
                 |x|^{2m-n} \ln{|x|}, & \text{if} \,\, n \,\, \text{is even and} \,\, 2m-n \geq 0  \\
                 |x|^{2m-n}, & \text{otherwise}
               \end{array} \right. . \label{func h}
\end{equation}

It is well known that if $a$ is a bounded function with compact support, its potential $b$, defined as $$b(x) = \int_{\mathbb{R}^{n}} h(x-y) a(y) dy,$$
is a locally bounded function and $\Delta^{m} b = a$ in the sense of distributions. For these potentials, we have the following

\begin{lemma} Let $a(\cdot)$ be an $(p(\cdot),p_{0},d)-$atom with $d = \max \{d_{p(\cdot)}, 2m-1\}$ and assume that $Q(x_0, r)$ is the cube containing the support of $a(.)$ in the definition of $(p(\cdot),p_{0},d)-$atom. If $$b(x) = \int_{\mathbb{R}^{n}} h(x-y) a(y) dy,$$ then for $|x - x_0| \geq \sqrt{n}r$ and every multiindex $\alpha$, there exists $c_{\alpha}$ such that $$\left| (\partial^{\alpha}b)(x) \right| \leq c_{\alpha} \, r^{2m+n} \| \chi_{Q} \|_{p(\cdot)}^{-1} |x - x_0|^{-n-\alpha}$$ holds.
\end{lemma}

\begin{proof} Since $a(\cdot)$ has vanishing moments up to the order $d \geq 2m-1$, we have
$$(\partial^{\alpha}b)(x) = \int_{Q(x_0, r)} \, (\partial^{\alpha}h)(x-y) a(y) \, dy$$
$$= \int_{Q(x_0, r)} \, \left[(\partial^{\alpha}h)(x-y) - \sum_{|\beta| \leq 2m-1} (\partial^{\alpha + \beta}h)(x-x_0) \frac{(x_0 -y)^{\beta}}{\beta !} \right] a(y) \, dy$$
$$= \int_{Q(x_0, r)} \, \left[\sum_{|\beta| = 2m} (\partial^{\alpha + \beta}h)(x- \xi) \frac{(x_0 -y)^{\beta}}{\beta !} \right] a(y) \, dy$$
where $\xi$ is a point between $y$ and $x_0$. If $|x - x_0| \geq \sqrt{n}r$ it follows that $|x - \xi| \geq \frac{|x - x_0|}{2}$ since $| x_0 - \xi| \leq \frac{\sqrt{n}}{2} r$. Taking into account that for $|\beta| = 2m$, $\partial^{\alpha + \beta}h$ is a homogeneous function of degree $-n-\alpha$, we obtain
$$| (\partial^{\alpha}b)(x) | \leq c \, r^{|\beta|} \int_{Q(x_0, r)} |x- \xi|^{-n-\alpha} |a(y)| dy$$
$$\leq c \, r^{|\beta|} \| a \|_{p_0} |Q|^{1-1/p_0} |x - x_0|^{-n-\alpha} \leq c \, r^{2m + n} \| \chi_{Q}\|_{p(\cdot)}^{-1} |x - x_0|^{-n-\alpha}.$$
\end{proof}

\begin{lemma} $($Lemma 8 in \cite{segovia}$)$ If $h$ is the kernel defined in $($\ref{func h}$)$ and $|\alpha|=2m$, then $(\partial^{\alpha}h)(x)$ is a $C^{\infty}$ homogeneous function of degree $-n$ on $\mathbb{R}^{n} \setminus \{0\}$, and $\int_{|x|=1} \, (\partial^{\alpha}h)(x) \, dx=0.$
\end{lemma}

We conclude this preliminaries with the following

\begin{proposition} Let $a(\cdot)$ be an $(p(\cdot),p_{0},d)-$atom with $d = \max \{d_{p(\cdot)}, 2m-1\}$, $p_0 >1$, and assume that $Q=Q(x_0, r)$ is the cube containing the support of $a(\cdot)$ in the definition of $(p(\cdot),p_{0},d)-$atom. If $b(x) = \int_{\mathbb{R}^{n}} h(x-y) a(y) dy$, then for all $x \in \mathbb{R}^{n}$ and all $0< \mu < 2m$
\begin{equation}
N_{q, 2m}(B; x) \lesssim \| \chi_{Q} \|_{p(\cdot)}^{-1} \left[M(\chi_{Q})(x) \right]^{\frac{2m + n/q - \mu}{n}} + \chi_{4\sqrt{n}Q}(x) M(a)(x) \label{N estimate}
\end{equation}
$$+ \chi_{4\sqrt{n}Q}(x) [M(M^{q}(a))(x)]^{1/q} + \chi_{4\sqrt{n}Q}(x) \sum_{|\alpha|=2m} T^{*}_{\alpha}(a)(x),$$
where $B$ is the class of $b$ in $E^{q}_{2m-1}$, $T^{*}_{\alpha}(a) (x) = \sup_{\epsilon >0} \left|\int_{|y|> \epsilon} \, (\partial^{\alpha}h)(y) a(x-y) \, dy\right|$ and $M$ is the maximal operator.
\end{proposition}

\begin{proof} For an $(p(\cdot), p_0, d)$ - atom $a$ satisfying the hypothesis of Proposition, we set
$$R(x, z)= b(x+z) - \sum_{|\alpha|\leq 2m-1} (\partial^{\alpha}b)(x) z^{\alpha}/ \alpha!$$
$$=b(x + z) - \sum_{|\alpha|\leq 2m-1} \left[\int_{Q(x_0, r)} \, (\partial^{\alpha}h)(x - y) a(y) \ dy \right] \frac{z^{\alpha}}{\alpha!}.$$

We shall estimate $|R(x,z)|$ considering the cases $|x - x_0| \geq 2\sqrt{n}r$ and $|x - x_0| < 2\sqrt{n}r$, and then we will obtain the estimate
(\ref{N estimate}).

\qquad

\textbf{Case:} $|x - x_0| \geq 2\sqrt{n}r$.

\qquad

For $|x - x_0| \geq 2\sqrt{n}r$, $|z| < \frac{1}{2}|x - x_0|$ and $0< \theta < 1$, we have $|x + \theta z - x_0| \geq |x - x_0| - |z| \geq \frac{1}{2}|x-x_0|\geq \sqrt{n}r$. Then, by the mean value theorem and Lemma 12, we get
\begin{equation}
|R(x,z)| \leq \sum_{|\alpha|=2m} |(\partial^{\alpha}b)(x+ \theta z)| \, \frac{|z|^{|\alpha|}}{\alpha!} \leq c \, \|\chi_{Q}\|^{-1}_{p(\cdot)} \left(\frac{r}{|x - x_0|}\right)^{2m+n} |z|^{2m}, \label{R1}
\end{equation}
for $|x - x_0| \geq 2\sqrt{n}r$ and $|z| < \frac{1}{2}|x - x_0|$.

Now, let $|z| \geq \frac{1}{2}|x - x_0|$. We have
$$|R(x,z)| \leq |b(x+z)| + \sum_{|\alpha|\leq 2m-1} |(\partial^{\alpha}b)(x)| |z|^{\alpha}/ \alpha!.$$

Since $|x - x_0| > 2 \sqrt{n} r$, by Lemma 12, and observing that $|z|/|x - x_0| > 1/2$, we have
$$|(\partial^{\alpha}b)(x)| \frac{|z|^{|\alpha|}}{\alpha!} \leq c \, \| \chi_{Q} \|_{p(\cdot)}^{-1} \left( \frac{r}{|x - x_0|} \right)^{2m+n} |z|^{2m}.$$

As for the other term, $|b(x+z)|$, we consider the cases $|x+z-x_0| > \sqrt{n}r$ and $|x+z-x_0|\leq \sqrt{n}r$. In the case $|x+z - x_0| > \sqrt{n}r$, we apply Lemma 12 for $\alpha = 0$, obtaining
$$|b(x+z)| \leq c \, \| \chi_{Q} \|_{p(\cdot)}^{-1} r^{2m+n} |x+z - x_0|^{-n}.$$
Then
\begin{equation}
|R(x,z)| \lesssim \| \chi_{Q} \|_{p(\cdot)}^{-1} r^{2m+n} |x+z - x_0|^{-n} + \| \chi_{Q} \|_{p(\cdot)}^{-1} \left( \frac{r}{|x - x_0|} \right)^{2m+n} |z|^{2m}, \label{R2}
\end{equation}
holds if $|x - x_0| \geq 2\sqrt{n}r$, $|z| \geq |x-x_0|/2$ and $|x+z - x_0| > \sqrt{n}r$.

For $|x+z-x_0|\leq \sqrt{n}r$, we consider the cases $n$ even and $2m-n \geq 0$, and $n$ odd or $2m-n < 0$. In the first case, we have that $|x|^{2m-n}$ is a polynomial of degree smaller than $2m$, and since $a(\cdot)$ is an $(p(\cdot), p_0, d)$ - atom, we have
$$|b(x+z)| = \left| \int_{Q(x_0,r)} |x+z-y|^{2m-n} \left[ \ln \left( |x+z-y| \right) - \ln(|x-x_0|) \right] a(y) dy \right|$$
$$\leq \|a \|_{p_0} \left(\int_{Q(x_0,r)} |x+z-y|^{(2m-n)p'_0} (\ln \left(|x- x_0|/ |x+z-y| \right))^{p'_0} \right)^{1/ p'_{0}}.$$
Since $|x+z-y|/ |x-x_0| < 1$, $\ln(t) \leq \mu^{-1} t^{\mu}$ for any $0 < \mu$ and $t \geq 1$, and observing that $Q(x_0,r) \subset \{y : |x+z-y|< \frac{3}{2} \sqrt{n}r \}$, we take $0< \mu < 2m-n$ if $2m-n >0$ or $0<\mu < n/p'_0$ if $2m-n=0$, to get
$$|b(x+z)| \lesssim \|\chi_{Q} \|_{p(\cdot)}^{-1} r^{n/p_0} |x - x_0|^{\mu} \left( \int_{|x+z-y|< \frac{3}{2} \sqrt{n}r} |x+z-y|^{(2m-n-\mu)p'_0} \right)^{1/p'_0}$$
$$\lesssim  r^{2m} \, \|\chi_{Q} \|_{p(\cdot)}^{-1} \left(\frac{r}{|x-x_0|} \right)^{-\mu}.$$
For the case $n$ odd and $2m-n \geq 0$, the proof is simpler since no logarithm appears, and the estimates obtained hold with $\mu=0$. For the case $2m - n<0$, we take $n/2m < s < p_0$ thus $0< (2m-n)s' + n$, from H\"older's inequality and Remark 11 it follows the estimate with $\mu=0$.
\\
Since $|x - x_0| \geq 2\sqrt{n}r$ we can conclude that
\begin{equation}
|R(x,z)| \lesssim r^{2m} \, \|\chi_{Q} \|_{p(\cdot)}^{-1} \left(\frac{r}{|x-x_0|} \right)^{-\mu} + \| \chi_{Q} \|_{p(\cdot)}^{-1} \left( \frac{r}{|x - x_0|} \right)^{2m+n} |z|^{2m}, \label{R3}
\end{equation}
for all $\mu > 0$, $|z| \geq |x - x_0|/2$ and $|x+z - x_0| \leq \sqrt{n}r$.

Let us the estimate
$$\delta^{-2m} \left( |Q(0, \delta)|^{-1} \int_{Q(0, \delta)} |R(x,z)|^{q} dz \right)^{1/q}, \,\,\,\, \delta > 0.$$
For them, we split the domain of integration into three subsets:

\qquad

$D_1 = \{ z \in Q(0, \delta) : |z| < |x-x_0|/2 \}$,

\qquad

$D_2 = \{z \in Q(0, \delta) : |z| \geq |x-x_0|/2, \, |x+z - x_0| > \sqrt{n}r \},$
\\\\
and

\qquad

$D_3 = \{z \in Q(0, \delta) : |z| \geq |x-x_0|/2, \, |x+z - x_0| \leq \sqrt{n}r \}$.
\\\\
According to the estimates obtained for $|R(x,z)|$ above, we use on $D_1$ the estimate (\ref{R1}), on $D_2$ the estimate (\ref{R2}) and on $D_3$ the estimate (\ref{R3}) to get
$$\delta^{-2m} \left( |Q(0, \delta)|^{-1} \int_{Q(0, \delta)} |R(x,z)|^{q} dz \right)^{1/q} \lesssim \| \chi_{Q} \|_{p(\cdot)}^{-1} \left(\frac{r}{|x-x_0|} \right)^{2m+n/q-\mu}.$$
Thus, for $0 < \mu < 2m$ we have
\begin{equation}
N_{q, 2m}(B,x) \lesssim \| \chi_{Q} \|_{p(\cdot)}^{-1} \left(\frac{r}{|x-x_0|} \right)^{2m+n/q-\mu} \lesssim \| \chi_{Q} \|_{p(\cdot)}^{-1} \left[ M(\chi_{Q})(x) \right]^{\frac{2m+n/q-\mu}{n}}, \label{Nq3}
\end{equation}
if $|x-x_0| \geq 2 \sqrt{n}r$.

\qquad

\textbf{Case:} $|x-x_0| < 2 \sqrt{n}r$.

\qquad

We have $$R(x,z) = \int \left[ h(x+z-y) - \sum_{|\alpha| \leq 2m-1} (\partial^{\alpha}h)(x-y) z^{\alpha}/\alpha! \right] a(y) dy$$
$$=\int_{|x-y|< 2 |z|} \, + \int_{|x-y|\geq 2 |z|} \, = I_1(x,z) + I_2(x,z).$$
Without loss of generality  we can assume that $y$ does not belong to the segment $[x,x+z]$, so we can write
$$U = h(x+z-y) - \sum_{|\alpha| \leq 2m-1} (\partial^{\alpha}h)(x-y) z^{\alpha}/\alpha!$$
$$= (2m-1) \sum_{|\alpha|=2m-1} (z^{\alpha}/\alpha!) \int_{0}^{1} (\partial^{\alpha}h)(x+tz-y) (1-t)^{2m-2} dt$$
$$+ \sum_{|\alpha| = 2m-1} (\partial^{\alpha}h)(x-y) z^{\alpha}/\alpha!$$
Since for $|\alpha|=2m-1$ the derivatives $\partial^{\alpha}h$ are homogeneous functions of degree $-n+1$, we get
$$|U| \leq c \, \left(\int_{0}^{1} |x+tz-y|^{-n+1} (1-t)^{2m-2} dt + |x-y|^{-n+1} \right) |z|^{2m-1}.$$
Observing that $|x-y| < 2|z|$ implies $|x+tz-y|<3|z|$, we obtain
$$|I_1(x,z)| \leq \int_{|x-y| < 2|z|} |U||a(y)| dy$$
$$\lesssim \, |z|^{2m-1}\int_{0}^{1} (1-t)^{2m-2} \left( \int_{|x+tz-y|<3|z|} |x+tz-y|^{-n+1} |a(y)| dy \right) dt$$
$$+ \, |z|^{2m-1} \int_{|x-y| < 2|z|} |x-y|^{-n+1} |a(y)| dy$$
$$= |z|^{2m-1}\int_{0}^{1} (1-t)^{2m-2} \left( \sum_{k=0}^{\infty}\int_{3^{-k}|z| \leq |x+tz-y|<3^{-(k-1)}|z|} |x+tz-y|^{-n+1} |a(y)| dy \right) dt$$
$$+ \, |z|^{2m-1} \sum_{k=0}^{\infty} \int_{2^{-k}|z| \leq |x-y| < 2^{-(k-1)}|z|} |x-y|^{-n+1} |a(y)| dy$$
$$\lesssim |z|^{2m} \left(\int_{0}^{1} (1-t)^{2m-2} \, M(a)(x+tz) \, dt + \, M(a)(x)\right).$$
To estimate $I_2(x,z)$, we write
$$U= \left[h(x+z-y) - \sum_{|\alpha| \leq 2m} (\partial^{\alpha}h)(x-y) z^{\alpha}/\alpha!\right] + \sum_{|\alpha| = 2m} (\partial^{\alpha}h)(x-y) z^{\alpha}/\alpha!$$
$$=U_1+U_2.$$
We have for some $0<s<1$
$$|U_1| \leq \sum_{|\alpha|=2m+1} |(\partial^{\alpha}h)(x+sz-y)| \, |z|^{2m+1}/\alpha!.$$
Since $|x-y|\geq 2|z|$ implies $|x+sz-y| \geq |x-y|/2$, and recalling that for $|\alpha|=2m+1$ the derivatives $\partial^{\alpha}h$ are homogeneous functions of degree $-n-1$, we get $|U_1|\leq c |x-y|^{-(n+1)} |z|^{2m+1}$. Therefore,
$$|I_2(x,z)| \lesssim |z|^{2m+1}\int_{|x-y| \geq 2|z|} |x-y|^{-(n+1)} |a(y)| dy + \left| \int_{|x-y| \geq 2|z|} U_2 \, a(y) dy \right|$$
$$\lesssim |z|^{2m} \left( M(a)(x) + \sum_{|\alpha|=2m} T_{\alpha}^{\ast}(a)(x) \right),$$
where $T^{*}_{\alpha}(a) (x) = \sup_{\epsilon >0} \left|\int_{|y|> \epsilon} \, (\partial^{\alpha}h)(y) a(x-y) \, dy\right|$.

Now, let us estimate $\delta^{-2m} \left( |Q(0, \delta)|^{-1} \int_{Q(0, \delta)} |I_1(x,z)|^{q} dz \right)^{1/q}$. We apply Minkowski's inequality for integrals to get
$$\delta^{-2m} \left( |Q(0, \delta)|^{-1} \int_{Q(0, \delta)} |I_1(x,z)|^{q} dz \right)^{1/q}$$
$$\lesssim \delta^{-2m} \int_{0}^{1} (1-t)^{2m-2} \, \left(|Q(0, \delta)|^{-1}\int_{Q(0, \delta)} |z|^{2mq} M^{q}(a)(x+tz) \, dz \right)^{1/q} \, dt$$
$$ + M(a)(x) \delta^{-2m}\left(|Q(0, \delta)|^{-1} \int_{Q(0,\delta)} |z|^{2mq} dz\right)^{1/q}$$
$$\lesssim [M(M^{q}(a))(x)]^{1/q} + M(a)(x).$$
It is easy to check that $$\delta^{-2m} \left( |Q(0, \delta)|^{-1} \int_{Q(0, \delta)} |I_2(x,z)|^{q} dz \right)^{1/q} \lesssim M(a)(x)
 + \sum_{|\alpha|=2m} T^{*}_{\alpha}(a)(x).$$
So
$$\delta^{-2m} \left( |Q(0, \delta)|^{-1} \int_{Q(0, \delta)} |R(x,z)|^{q} dz \right)^{1/q} \lesssim [M(M^{q}(a))(x)]^{1/q} + M(a)(x)$$
$$ + \sum_{|\alpha|=2m} T^{*}_{\alpha}(a)(x).$$
This estimate is global, in particular we have
\begin{equation}
N_{q,2m}(B;x) \lesssim [M(M^{q}(a))(x)]^{1/q} + M(a)(x) + \sum_{|\alpha|=2m} T^{*}_{\alpha}(a)(x), \label{Nq4}
\end{equation}
for all $x \in 4\sqrt{n}Q$. Finally, the estimates (\ref{Nq3}) and (\ref{Nq4}) for $N_{q,2m}(B;x)$ allow us to obtain (\ref{N estimate}).
\end{proof}

\section{Proofs of the results}

\qquad

\textit{Proof of Theorem 1.} Let $F \in \mathcal{H}^{p(\cdot)}_{q, \, 2m}(\mathbb{R}^{n})$. Since $N_{q, 2m}(F; x)$ is finite p.p.$x \in \mathbb{R}^{n}$, by corollary 3 in \cite{segovia}, the unique representative $f$ of $F$ satisfying $\eta_{q, 2m}(f; x)=N_{q, 2m}(F; x)$ is a function in $L^{q}_{loc}(\mathbb{R}^{n}) \cap \mathcal{S}'(\mathbb{R}^{n})$. Thus, if $\phi \in \mathcal{S}(\mathbb{R}^{n})$ and $\int \phi(x) \, dx \neq 0$, from Lemma 6 in \cite{segovia} we get
$$M_{\phi}(\Delta^{m}F)(x) \leq c \, p_{2m+n}(\phi) N_{q, \, s}(F; x).$$
Thus $\Delta^{m}F \in H^{p(\cdot)}(\mathbb{R}^{n})$ and $\| \Delta^{m}F \|_{H^{p(\cdot)}} \leq c \, \| F \|_{\mathcal{H}^{p(\cdot)}_{q, \, 2m}}.$

Now we shall see that the operator $\Delta^{m}$ is onto. By Theorem 4.6 in \cite{nakai}, given $f \in H^{p(\cdot)}(\mathbb{R}^{n})$ there exists a sequence of nonnegative numbers $\{ k_j \}_{j=1}^{\infty}$ and a sequence of cubes $\{Q_j \}_{j=1}^{\infty}$ and $(p(\cdot), p_0, d)$ - atoms $a_j$ supported on $Q_j$, such that $f= \sum_{j=1}^{\infty} k_j a_j$ and
$$\mathcal{A}\left( \left\{ k_{j}\right\}
_{j=1}^{\infty },\left\{ Q_{j}\right\} _{j=1}^{\infty },p(\cdot)\right) \lesssim \|f \|_{H^{p(\cdot)}}.$$ For each $j \in \mathbb{N}$ we put
$b_j(x)= \int_{\mathbb{R}^{n}} h(x-y) a_j(y) dy,$ from Proposition 15 we have
$$N_{q, 2m}(B_j; x)\lesssim  \| \chi_{Q_j} \|_{p(\cdot)}^{-1} \left[M(\chi_{Q_j})(x) \right]^{\frac{2m + n/q - \mu}{n}} + \chi_{4\sqrt{n}Q_j}(x) M(a_j)(x)$$
$$+ \chi_{4\sqrt{n}Q_j}(x) [M(M^{q}(a_j))(x)]^{1/q} + \chi_{4\sqrt{n}Q_j}(x) \sum_{|\alpha|=2m} T^{*}_{\alpha}(a_j) (x).$$
So
$$\sum_{j=1}^{\infty} k_j N_{q, 2m}(B_j; x) \lesssim \sum_{j=1}^{\infty} k_j \frac{\left[M(\chi_{Q_j})(x) \right]^{\frac{2m + n/q - \mu}{n}}}{\| \chi_{Q_j} \|_{p(\cdot)}} + \sum_{j=1}^{\infty} k_j \chi_{4\sqrt{n}Q_j}(x) M(a_j)(x) $$ $$\sum_{j=1}^{\infty} k_j \chi_{4 \sqrt{n} Q_j}(x) [M(M^{q}(a_j))(x)]^{1/q} + \sum_{j=1}^{\infty} k_j \chi_{4 \sqrt{n} Q_j}(x)  \sum_{|\alpha|=2m} T^{*}_{\alpha}(a_j) (x)$$ $$= I + II + III + IV.$$
To study $I$, by hypothesis, we have that $(2m + n/q) \underline{p} > n $, thus we can choose  $0 < \mu < 2m$ such that $(2m + n/q - \mu) \underline{p} > n $. Then
$$\|I\|_{p(\cdot)} = \left\| \sum_{j=1}^{\infty} k_j \| \chi_{Q_j} \|_{p(\cdot)}^{-1} \left[M(\chi_{Q_j})(\cdot) \right]^{\frac{2m + n/q - \mu}{n}} \right\|_{p(\cdot)}$$
$$= \left\| \left\{ \sum_{j=1}^{\infty} k_j \| \chi_{Q_j} \|_{p(\cdot)}^{-1} \left[M(\chi_{Q_j})(\cdot) \right]^{\frac{2m + n/q - \mu}{n}} \right\}^{\frac{n}{2m + n/q - \mu}} \right\|_{\frac{2m + n/q - \mu}{n}p(\cdot)}^{\frac{2m + n/q - \mu}{n}}$$
$$\lesssim \left\| \left\{ \sum_{j=1}^{\infty} k_j \| \chi_{Q_j} \|_{p(\cdot)}^{-1} \chi_{Q_j}(\cdot)  \right\}^{\frac{n}{2m + n/q - \mu}} \right\|_{\frac{2m + n/q - \mu}{n}p(\cdot)}^{\frac{2m + n/q - \mu}{n}} = \left\| \sum_{j=1}^{\infty} k_j \| \chi_{Q_j} \|_{p(\cdot)}^{-1} \chi_{Q_j}(\cdot) \right\|_{p(\cdot)}$$
$$\lesssim \left\| \left\{ \sum_{j=1}^{\infty} \left( k_j \| \chi_{Q_j} \|_{p(\cdot)}^{-1} \chi_{Q_j}(\cdot) \right)^{\underline{p}}  \right\}^{1/ \underline{p}} \right\|_{p(\cdot)} = \mathcal{A}\left( \left\{ k_{j}\right\}
_{j=1}^{\infty },\left\{ Q_{j}\right\} _{j=1}^{\infty },p(\cdot)\right) \lesssim \|f \|_{H^{p(\cdot)}},$$
where the first inequality follows from Lemma 2.4 in \cite{nakai}, since $(2m + n/q - \mu) \underline{p} > n $. The embedding $l^{\underline{p}} =l^{\min \{ p_{-}, 1 \}} \hookrightarrow l^{1}$ gives the second inequality.

\qquad

To study $II$, we have that the maximal operator $M$ is bounded on $L^{p_0}$ for each $1 < p_0 < \infty$, thus
$$\|  M(a_j) \|_{L^{p_0}(4\sqrt{n}Q_j)} \lesssim \| a_j \|_{p_0} \lesssim \frac{ |Q_j |^{1/p_0}}{\| \chi_{Q_j} \|_{p(\cdot)}} \lesssim \frac{ |4\sqrt{n}Q_j |^{1/p_0}}{\| \chi_{4\sqrt{n}Q_j} \|_{p(\cdot)}},$$
the third inequality holds since the quantities $\| \chi_{4\sqrt{n}Q_j} \|_{p(\cdot)}$ and $\| \chi_{Q_j} \|_{p(\cdot)}$ are comparable (see Lemma 2.2 in \cite{nakai}). Now if
$\| M(a_j) \|_{L^{p_0}(4\sqrt{n}Q_j)} \neq 0$ we obtain
$$\| II\|_{p(\cdot)} = \left\| \sum_{j=1}^{\infty} k_j \chi_{4 \sqrt{n} Q_j}(\cdot) M(a_j) (\cdot)  \right\|_{p(\cdot)}$$
$$\lesssim  \left\| \sum_{j=1}^{\infty} k_j \frac{ \chi_{4 \sqrt{n} Q_j}(\cdot) M(a_j) (\cdot) |4\sqrt{n}Q_j |^{1/p_0}}{\| M(a_j) \|_{L^{p_0}(4\sqrt{n}Q_j)} \, \| \chi_{4\sqrt{n}Q_j} \|_{p(\cdot)}}  \right\|_{p(\cdot)}$$
$$\lesssim \left\| \left\{\sum_{j=1}^{\infty} \left(k_j \frac{ \chi_{4 \sqrt{n} Q_j}(\cdot) M(a_j) (\cdot) |4\sqrt{n}Q_j |^{1/p_0}}{\| M(a_j) \|_{L^{p_0}(4\sqrt{n}Q_j)} \, \| \chi_{4\sqrt{n}Q_j}\|_{p(\cdot)}} \right)^{\underline{p}} \right\}^{1/\underline{p}}   \right\|_{p(\cdot)}$$
now take $p_0 > 1$ sufficiently large such that $\delta = \frac{1}{p_0}$ satisfies the hypothesis of Lemma 4.11 in \cite{nakai} to get
$$\lesssim \mathcal{A}\left( \left\{ k_{j}\right\}
_{j=1}^{\infty },\left\{4 \sqrt{n} Q_{j}\right\} _{j=1}^{\infty },p(\cdot)\right) \lesssim  \mathcal{A}\left( \left\{ k_{j}\right\}
_{j=1}^{\infty },\left\{ Q_{j}\right\} _{j=1}^{\infty },p(\cdot)\right) \lesssim \|f \|_{H^{p(\cdot)}}.$$

\qquad

To study $III$, we consider $2q < p_0$, then by H\"older's inequality and Remark 11 we obtain
$$\| [M(M^{q}(a_j))]^{1/q} \|_{L^{q}(4 \sqrt{n}Q_j)} \lesssim \| M(M^{q}(a_j)) \|_{2}^{1/q} |4\sqrt{n}Q_j|^{1/2q}$$
$$\lesssim \| M^{q}(a_j) \|_{2}^{1/q} |4\sqrt{n}Q_j|^{1/2q} = \| M(a_j) \|_{2q} |4\sqrt{n}Q_j|^{1/2q}$$
$$\lesssim \| a_j \|_{2q} |4\sqrt{n}Q_j|^{1/2q} \lesssim \frac{|4\sqrt{n}Q_j|^{1/q}}{\| \chi_{Q_j} \|_{p(\cdot)}} \lesssim \frac{ |4\sqrt{n}Q_j |^{1/q}}{\| \chi_{4\sqrt{n}Q_j} \|_{p(\cdot)}}.$$
Now, if $\| [M(M^{q}(a_j))]^{1/q} \|_{L^{q}(4 \sqrt{n}Q_j)} \neq 0$ and taking $q > 1$ sufficiently large such that $\delta = \frac{1}{q}$ satisfies the hypothesis of Lemma 4.11 in \cite{nakai}, a similar computation to done in the estimate of $II$ allows us to get
$$\|III\|_{p(\cdot)} = \left\|  \sum_{j=1}^{\infty} k_j \chi_{4 \sqrt{n} Q_j}(\cdot) [M(M^{q}(a_j))(\cdot)]^{1/q} \right\|_{p(\cdot)}$$
$$\lesssim \left\| \left\{\sum_{j=1}^{\infty} \left(k_j \frac{ \chi_{4 \sqrt{n} Q_j}(\cdot) [M(M^{q}(a_j))(\cdot)]^{1/q} |4\sqrt{n}Q_j |^{1/q}}{\| [M(M^{q}(a_j))(\cdot)]^{1/q} \|_{L^{p_0}(4\sqrt{n}Q_j)} \, \| \chi_{4\sqrt{n}Q_j}\|_{p(\cdot)}} \right)^{\underline{p}} \right\}^{1/\underline{p}}   \right\|_{p(\cdot)} \lesssim \|f\|_{H^{p(\cdot)}}.$$

\qquad

To study $IV$, by Lemma 14 and Theorem 4 in \cite{stein} p. 42, we have that the operator $T_{\alpha}^{*}$ is bounded on $L^{p_0}$ for each $1 < p_0 < \infty$, so $$\| T_{\alpha}^{*} (a_j) \|_{L^{p_0}(4\sqrt{n}Q_j)} \lesssim \| a_j \|_{p_0} \lesssim \frac{ |Q_j |^{1/p_0}}{\| \chi_{Q_j} \|_{p(\cdot)}} \lesssim \frac{ |4\sqrt{n}Q_j |^{1/p_0}}{\| \chi_{4\sqrt{n}Q_j} \|_{p(\cdot)}}.$$
Now, if $\| T_{\alpha}^{*} (a_j) \|_{L^{p_0}(4\sqrt{n}Q_j)} \neq 0$ we take $p_0 > 1$ sufficiently large such that $\delta = \frac{1}{p_0}$ satisfies the hypothesis of Lemma 4.11 in \cite{nakai}, once again a similar computation to done in the estimate of $II$, we get
$$\| IV\|_{p(\cdot)} = \left\| \sum_{j=1}^{\infty} k_j  \sum_{|\alpha|=2m} \chi_{4 \sqrt{n} Q_j}(\cdot) T^{*}_{\alpha}(a_j) (\cdot)  \right\|_{p(\cdot)}$$
$$\lesssim \sum_{|\alpha|=2m} \left\| \left\{\sum_{j=1}^{\infty} \left(k_j \frac{ \chi_{4 \sqrt{n} Q_j}(\cdot) T^{*}_{\alpha}(a_j) (\cdot) |4\sqrt{n}Q_j |^{1/p_0}}{\| T_{\alpha}^{*} (a_j) \|_{L^{p_0}(4\sqrt{n}Q_j)} \, \| \chi_{4\sqrt{n}Q_j}\|_{p(\cdot)}} \right)^{\underline{p}} \right\}^{1/\underline{p}}   \right\|_{p(\cdot)} \lesssim \|f \|_{H^{p(\cdot)}}.$$

Thus we have
$$\left\| \sum_{j=1}^{\infty} k_j N_{q, 2m}(B_j; \cdot) \right\|_{p(\cdot)} \lesssim \|f \|_{H^{p(\cdot)}}.$$
By Lemma 7, we obtain $\rho_{p(\cdot)}\left( \sum_{j=1}^{\infty} k_j N_{q, 2m}(B_j; \cdot) \right) < \infty$. Hence
\begin{equation}
\sum_{j=1}^{\infty} k_j N_{q, 2m}(B_j; x) < \infty \,\,\,\,\,\, \text{p.p.}x \in \mathbb{R}^{n} \label{Nq}
\end{equation}
and
\begin{equation}
\rho_{p(\cdot)}\left( \sum_{j=M+1}^{\infty} k_j N_{q, 2m}(B_j; \cdot) \right) \rightarrow 0, \,\,\,\, \text{as} \,\,  M \rightarrow \infty  \label{Nq2}.
\end{equation}
From (\ref{Nq}) and Lemma 6, there exists a function $F$ such that $\sum_{j=1}^{\infty} k_j B_j = F$ in $E^{q}_{2m-1}$ and
$$N_{q, 2m} \left( \left(F - \sum_{j=1}^{M} k_j B_j \right) ; \, x \right) \leq c \, \sum_{j=M+1}^{\infty} k_j N_{q, 2m}(B_j; x).$$
This estimate together with (\ref{Nq2}) and Lemma 8 implies
$$\left\| F - \sum_{j=1}^{M} k_j B_j \right\|_{\mathcal{H}^{p(\cdot)}_{q,2m}} \rightarrow 0, \,\,\,\, \text{as} \,\,  M \rightarrow \infty.$$
By proposition 9, we have that $F \in \mathcal{H}^{p(\cdot)}_{q,2m}(\mathbb{R}^{n})$ and $F = \sum_{j=1}^{\infty} k_j B_j$ in $\mathcal{H}^{p(\cdot)}_{q,2m}(\mathbb{R}^{n})$. Since $\Delta^{m}$ is a continuous operator from $\mathcal{H}^{p(\cdot)}_{q,2m}(\mathbb{R}^{n})$ into $H^{p(\cdot)}(\mathbb{R}^{n})$, we get
$$\Delta^{m}F = \sum_j k_j \Delta^{m} B_j = \sum_j k_j a_j = f,$$
in $H^{p(\cdot)}(\mathbb{R}^{n})$. This shows that $\Delta^{m}$ is onto $H^{p(\cdot)}(\mathbb{R}^{n})$. Moreover,
$$\| F\|_{\mathcal{H}^{p(\cdot)}_{q,2m}} = \left\| \sum_{j=1}^{\infty} k_j B_j \right\|_{\mathcal{H}^{p(\cdot)}_{q,2m}} \lesssim \left\| \sum_{j=1}^{\infty} k_j N_{q, 2m}(B_j; \cdot) \right\|_{p(\cdot)} \lesssim \|f \|_{H^{p(\cdot)}} = \| \Delta^{m} F \|_{H^{p(\cdot)}}.$$

\qquad

For to conclude the proof, we will show that the operator $\Delta^{m}$ is injective on $\mathcal{H}^{p(\cdot)}_{q, \, 2m}$. Let $\mathcal{O} = \{ x : N_{q, 2m}(F; x) > 1\}$. The set $\mathcal{O}$ is open because of that $N_{q, 2m}(F; \cdot)$ is lower semicontinuous. We take a constant $r>0$  such that $r \geq \{q, p_{+}\}$. Since $N_{q, 2m}(F; \cdot) \in L^{p(\cdot)}(\mathbb{R}^{n})$ it follows that $|\mathcal{O}|$ is finite and $N_{q, 2m}(F; \cdot) \in L^{r}(\mathbb{R}^{n} \setminus \mathcal{O})$ thus $F \in \mathcal{N}^{r, q}_{2m}$, (for the definition of the space $\mathcal{N}^{r, q}_{2m}$ see p. 564 in \cite{calderon}). Finally, the injectivity of the operator $\Delta^{m}$ it follows from Lemma 9 in \cite{calderon}. $\square$

\qquad

\textit{Proof of Theorem 2.} Let $F \in \mathcal{H}^{p(\cdot)}_{q, \, 2m}(\mathbb{R}^{n})$ and assume $F \neq 0$. Then there exists $g \in F$ that is not a polynomial of degree less or equal to $2m-1$. It is easy to check that there exist a positive constant $c$ and a cube $Q = Q(0,r)$ with $r>1$ such that
$$\int_{Q} |g(y) - P(y)| dy \geq c > 0,$$
for every $P \in \mathcal{P}_{2m-1}$.

Let $x$ be a point such that $|x| > \sqrt{n}r$ and let $\delta= 4 |x|$. Then $Q(0,r) \subset Q(x, \delta)$. If $f \in F$, then $f = g - P$ for some $P \in \mathcal{P}_{2m-1}$ and
$$\delta^{-2m}|f|_{q, Q(x, \delta)} \geq c |x|^{-2m-n/q}.$$
So $N_{q,2m}(F;x) \geq c \, |x|^{-2m-n/q}$, for $|x| > \sqrt{n}r$. Since $p_{+} \leq n(2m+n/q)^{-1}$, we have
$$\rho_{p(\cdot)}(N_{q,2m}(F; \cdot)) \geq c \, \int_{|x| > \sqrt{n}r} |x|^{-(2m+n/q)p_{+}} \, dx = \infty.$$
In view of Lemma 7, we get a contradiction. Thus $\mathcal{H}^{p(\cdot)}_{q, 2m}(\mathbb{R}^{n}) = \{0\}$, if $p_{+} \leq n(2m+n/q)^{-1}$. $\square$

\qquad

\begin{remark} Theorem 1 does not hold in general for $\mathcal{H}^{p(\cdot)}_{q,\gamma}(\mathbb{R})$ when $\gamma$ is not a natural number. In \cite{sheldy}, S. Ombrosi gives an example which Theorem 1 is not true for $\mathcal{H}^{p}_{q,\gamma}(\mathbb{R})$ and the operator $\left(\frac{d}{dx}\right)^{\gamma}$, when
$0 < \gamma < 1 $ and $(\gamma + 1/q)p>1$.
\end{remark}

\end{document}